\documentclass[a4paper,reqno]{amsart}

\usepackage{amsmath, amssymb, amsthm, hyperref, url,cprotect,mathtools}

\usepackage[T1]{fontenc}
\usepackage[utf8]{inputenc}
\numberwithin{equation}{section}

\date{13 December 2022}

\title[Generators of the group of modular units
for $\Gamma^1(N)$ over $\QQ$]{Generators of the group of modular
units for $\Gamma^1(N)$ over the rationals}

\author{Marco Streng}
\thanks{Universiteit Leiden, Niels Bohrweg 1, 2333 CA  Leiden, The Netherlands.
\url{streng@math.leidenuniv.nl},
\url{http://pub.math.leidenuniv.nl/~strengtc/}
The author would like to thank
Peter Bruin, Maarten Derickx,
P\i{}nar K\i{}l\i{}\c{c}er and Mark van Hoeij
for helpful discussions
and the anonymous referee for helpful suggestions for improving the exposition.
}

\newtheorem{theorem}{Theorem}[section]

\theoremstyle{definition}
\newtheorem{remark}[theorem]{Remark}
\newtheorem{example}[theorem]{Example}

\newtheorem{proposition}[theorem]{Proposition}

\newtheorem{corollary}[theorem]{Corollary}
\newtheorem{lemma}[theorem]{Lemma}

\newcommand{\ZZ}{\mathbf{Z}}
\newcommand{\QQ}{\mathbf{Q}}
\newcommand{\CC}{\mathbf{C}}
\newcommand{\RR}{\mathbf{R}}
\renewcommand{\AA}{\mathbf{A}}
\newcommand{\SL}{\mathrm{SL}}
\newcommand{\HH}{\mathbf{H}}
\newcommand{\PP}{\mathbf{P}}

\begin{document}

\begin{abstract}
We give two explicit sets of generators of the
group of invertible regular functions over~$\QQ$
on the modular curve~$Y^1(N)$.

The first set of generators
is very surprising.
It is essentially the set of defining
equations of $Y^1(k)$ for $k\leq N/2$
when all these modular curves
are simultaneously embedded into
the affine plane,
and this proves a conjecture of
Derickx and Van Hoeij~\cite{derickx-vanhoeij}.
This set of generators
is an elliptic divisibility sequence
in the sense that it
satisfies the same recurrence relation
as the elliptic division polynomials.

The second set of generators
is explicit in terms of classical analytic functions
known as Siegel functions.
This is both a generalization and a converse
of a result of
Yang~\cite{yifanyang,yifanyang2}.
\end{abstract}

\maketitle

\section{Introduction}
\label{sec:intro}

Let $N\geq 1$ be an integer.
The modular curve $Y^1(N)$ is a smooth, affine,
geometrically irreducible
algebraic curve over~$\QQ$,
often also denoted by $Y_1(N)$.
It has the following property:
For every field~$K$
of characteristic zero,
if $N\geq 4$ or $K$ is algebraically closed,
then we have
\begin{align*}
&Y^1(N)(K) = \\
&\{(E, P) : \text{$E$ is an elliptic curve over $K$ and $P\in E(K)$ has order $N$}\}/\cong.
\end{align*}
Here ``='' denotes a functorial Galois-equivariant bijection,
which we use to identify the left and right hand side;
and we write $(E_1,P_1)\cong (E_2,P_2)$
when there is an isomorphism $\phi : E_1\rightarrow E_2$
with $\phi(P_1)=P_2$.

Our object of study is the 
group of \emph{modular units on $Y^1(N)$},
that is, the
unit group $\mathcal{O}(Y^1(N))^*$
of the ring $\mathcal{O}(Y^1(N))$ of regular
functions over $\QQ$ on $Y^1(N)$.
The curve $Y^1(N)$ has a smooth compactification
$X^1(N)$, and the group
$\mathcal{O}(Y^1(N))^*$
equals the group of meromorphic functions over $\QQ$
on $X^1(N)$ with divisor supported on the set
$X^1(N)\setminus Y^1(N)$ of \emph{cusps}.

The \emph{Tate normal form} (Section~\ref{sec:tateform})
gives an embedding
$Y^1(N)\hookrightarrow\AA^2$
for every $N\geq 4$, 
with the point $(B,C)\in\AA^2$ corresponding to the curve
\begin{equation}\label{eq:tateform1}
	E:Y^2 + (1-C)XY - BY = X^3 - BX^2\quad\mbox{and point}\quad P=(0,0).
\end{equation}
Our first main result is as follows.
\begin{theorem}[Conjecture 1 of Derickx and Van Hoeij~\cite{derickx-vanhoeij}]\label{thm:mainthm}
For all $k\geq 4$, let $F_k\in\QQ[B,C]$ be the defining polynomial of
$Y^1(k)$ inside~$\AA^2$.
Then for all $N\geq 4$, the group $\mathcal{O}(Y^1(N))^*$
is $\QQ^*$ times the free abelian group on
$B$, $D$, $F_4$, $F_5,\ldots,F_{\lfloor N/2\rfloor +1}$,
where $D\in\QQ[B,C]$ is the discriminant of~\eqref{eq:tateform1}.
\end{theorem}
The functions $F_{k}$ are given
in terms of a recurrence relation, which we recall
in Remark~\ref{rem:recurrence}.

The theorem is interesting for a number of reasons.
First of all, Derickx and Van Hoeij~\cite{derickx-vanhoeij}
already used the functions in the theorem
in order to compute the gonality of $Y^1(N)$ for all positive
integers
$N\leq 40$ and to give an upper bound on the gonality for $N\leq 250$.
Our theorem helps explain why their method was successful.

Moreover, they found that the gonality is often achieved by functions
from this set of generators. In particular,
these functions are ``small'' functions in some sense, which
we therefore hope are 
suitable for finding ``small'' models of
modular curves $Y^1(N)$.
Finding such small models directly in terms of another algebraic
model has the advantage that no approximate numerics
(such as floating point
numbers or truncated power series)
are needed in producing these models,
as would be the case when using theta functions or Siegel functions
directly or using modular forms.

Thirdly, as we will see in Section~\ref{sec:tateform},
the functions $F_k$ are the primitive divisors
of an \emph{elliptic divisibility sequence (EDS)} $P_1,P_2,P_3,\ldots$
over
the ring $\QQ[B,C]$,
which is in a way the \emph{universal} EDS
as it comes from the Tate normal form.
In line with
Ingram-Mah\'e-Silverman-Stange-Streng~\cite{edsff}
and Naskręcki~\cite{Naskrecki_NewYork},
all but finitely many terms $P_{k}$ have a primitive
divisor. In fact, we prove that all terms $P_{k}$
with $k>3$ have a \emph{unique} primitive divisor~$F_{k}$.

Finally, an explicit basis of the unit group
could be useful for computing cuspidal divisor
class groups similarly to~\cite{yifanyang2}.
 
The proof proceeds by first linking the functions $P_k$ to
classical analytic Siegel functions, and then
observing how a proof of Kubert and Lang for $Y(N)$
can be much simplified and strengthened when
applying it to $Y^1(N)$.
Our proof can be read without knowing
the proof of Kubert and Lang,
and can be seen as an introduction
into their methods due to the disappearance
of complications that arise in their proof.

We prove the main theorem using modular forms
over~$\CC$.
Let $\HH\subset\CC$ be the standard upper half
plane, write $\HH^*=\HH\cup \PP^1(\QQ)\subset\PP^1(\CC)$,
and write
\begin{equation}\label{eq:gamma1n}
\Gamma^1(N) = \left\{\begin{psmallmatrix}a & b \\ c & d\end{psmallmatrix}\in \SL_2(\ZZ) :
      b\equiv 0, a\equiv d\equiv 1\ \mathrm{mod}\ N\right\}.
\end{equation}
Recall the natural complex analytic isomorphism
\begin{align}\label{eq:ourparam}
 \Gamma^1(N)\backslash \HH &\longrightarrow  Y^1(N)(\CC) \\
 \tau &\longmapsto (\CC / \Lambda_\tau, \tau/N\ \mathrm{mod}\ \Lambda_\tau),
 \nonumber
\end{align}
where $\Lambda_\tau = \tau\ZZ+\ZZ$.
(If the reader is used to another parametrization,
see Remarks \ref{rem:altparam2} and~\ref{rem:nonstandard} below.)
The functions on $X^1(N)$ defined over~$\QQ$
correspond exactly to the meromorphic
functions on $\Gamma^1(N)\backslash \HH^*$
whose $q$-expansions at $\infty$ are rational,
that is, in~$\QQ((q^{1/N}))$ with
$q^a = \exp(2\pi i a \tau).$

The group $\mathcal{O}(Y^1(N))^*$ therefore
equals the group of meromorphic functions on
$\Gamma^1(N)\backslash \HH^*$ with rational $q$-expansion
and divisor supported 
on~$\PP^1(\QQ)$.

Our second main result is as follows.
For positive integers $k\leq N/2$, let $H_k$
be the \emph{Siegel function} given by
(see also \eqref{eq:Hk2})
\begin{equation}\label{eq:Hk} H_k(\tau) = iq^{\frac{1}{2}\left((k/N)^2-k/N+\frac{1}{6}\right)}(1-q^{k/N})\prod_{n=1}^{\infty}
(1-q^{n+k/N})(1-q^{n-k/N}).\end{equation}
\begin{theorem}\label{thm:mainthmtwo}
Let $$S = \left\{ \prod_{k=1}^{\lfloor N/2\rfloor}
H_{k}^{e(k)} :
\begin{array}{rcl}
\forall_k\ \  e(k)&\in&\ZZ,\\
 \sum_{k} e(k)&\in& 12\ZZ,\\
\sum_k k^2 e(k) &\in& \mathrm{gcd}(N,2)N\ZZ
\end{array}\right\}.$$
Then $S$ is free abelian of rank~$\lfloor N/2\rfloor $
and satisfies $\mathcal{O}(Y^1(N))^* = \QQ^*\cdot S$.
\end{theorem}
\begin{remark}\label{rem:kubertlangremark}
Kubert and Lang
have results similar to Theorem~\ref{thm:mainthmtwo} for
the curve $Y(N)$
(Theorems 1 and~2 of \cite{kubert-lang-IV}; 
alternatively Theorems 1.1 and 1.2 in Chapter~4 of~\cite{kubert-lang-book}).
Indeed, the results of loc.~cit.\ can be combined
into an analogue of our Theorem~\ref{thm:mainthmtwo}
for $\mathcal{O}(Y(N)_{\CC})$, but for most~$N$
their result is only
`up to power of two index'.
For details, see Theorem 1.3 in Chapter 4 of \cite{kubert-lang-book}
and the text below it. See also Kubert~\cite{kubert-square-root}.
\end{remark}

\begin{remark}
	Theorem~\ref{thm:mainthmtwo} gives both a strengthening and a
	converse of Corollary~3 of Yang~\cite{yifanyang}.
	Indeed, loc.~cit.~gives
	the inclusion $S'\subset \QQ(Y^1(N))$
	if $S'\subset S$ is defined by
	the additional hypotheses
	$\sum_{k} k e(k)\in 2\ZZ$ and $\sum_{k} k^2 e(k)\in 2N\ZZ$.

	Theorems 1--5 of \cite{yifanyang2} give the analogue of
	Theorem~\ref{thm:mainthmtwo}
	if one restricts to the functions with 
	divisors supported on 
	cusps $\frac{x}{y}$ for $\gcd(x,N)=1$. 
	
	The dictionary between our functions
	and the functions of
	\cite{yifanyang, yifanyang2} is given in Remark~\ref{rem:nonstandard}
	below.
	And in fact, with the conventions of Remark~\ref{rem:nonstandard}
     the functions of \cite{yifanyang2}
     are those with divisors supported on cusps
     $\frac{x}{Ny}$ for $\gcd(x,N)=1$.
	% Theorem 1, just generators, if N>3 prime
	% Theorem 2, basis if N = p^k with p>2 prime and k>1
	% Theorem 3, basis if N = 2^k with k>2
	% Theorem 4, N composite squarefree
	% Theorem 5 (looks general, but notation is "as above", so I'm not sure)
\end{remark}

\subsection{Overview and methods}
\label{sec:overview}

Our proof consists of two parts.
The first part is Section~\ref{sec:relating}, which relates the functions
of Theorems \ref{thm:mainthm} and~\ref{thm:mainthmtwo} via explicit
expressions in both directions.
We use formulas and techniques
from the theory of elliptic divisibility sequences
to relate division polynomials with the Weierstrass sigma function.

The second part is Section~\ref{sec:powerseries},
in which we show that our functions indeed generate the full group.
As in Kubert-Lang~\cite{kubert-lang-IV}, one of the key ideas is to
use the fact that every modular form with a rational
$q$-expansion can be scaled to have an integer $q$-expansion.
Together with Gauss' Lemma for power series with bounded
denominators, this will show that if~$g^l$ is in our group
for a modular function~$g$,
then so is~$g$ itself.
We show that this idea works even better in the
case of $\Gamma^1(N)$ over~$\QQ$
than in the case of \cite{kubert-lang-IV},
yielding results that are less
general, but stronger, simpler and more elegant than
the results of~\cite{kubert-lang-IV}.
A detailed overview of this part of the proof
is given at the beginning
of Section~\ref{sec:powerseries}.

Before we start the proof, Section~\ref{sec:functions}
gives precise definitions of
the functions appearing in Theorems \ref{thm:mainthm} and~\ref{thm:mainthmtwo}.

After the proof is finished, we give two results
that we get for free from our methods.
In Section~\ref{sec:ring}, we give 
generators of the \emph{ring} $\mathcal{O}(Y^1(N))$
instead of generators of the unit group,
and in Section~\ref{sec:theta},
we express the generators of the unit group in terms of theta functions.

\section{The functions appearing in the main results}
\label{sec:functions}

\subsection{The Tate normal form}
\label{sec:tateform}

Let $E$ be an elliptic curve over a field $K$
and $P\in E(K)$ a point of order $>3$
(possibly non-torsion).
\begin{lemma}[Tate normal form]\label{lem:tateform}
Every pair $(E, P)$ as above is isomorphic
to a unique pair of the form
\begin{equation}\label{eq:tateform}
E:Y^2 + (1-C)XY - BY = X^3 - BX^2,\quad P=(0,0)
\end{equation}
for $B, C\in K$ with 
\[D := B^{3} \cdot (C^{4} - 8 B C^{2} - 3 C^{3} + 16 B^{2} - 20 B C + 3 C^{2} + B -  C)\not=0.\]
Conversely, for every pair $B, C\in K$ with
$D\not=0$, equation~\eqref{eq:tateform} gives
a pair $(E, P)$.
\end{lemma}
\begin{proof}
Given $(E, P)$,
start with a general Weierstrass equation
\begin{equation}
\label{eq:weierstrass}
Y^2 + A_1 XY + A_3 Y = X^3 + A_2 X^2 + A_4 X + A_6.
\end{equation}
As $P$ does not have order~$1$, it is affine,
and we translate $P$ to $(0,0)$ yielding $A_6=0$.
As $P$ does not have order~$2$, we have $A_3\not=0$,
and we add $(A_4/A_3)X$ to~$Y$ to get $A_4=0$.
As $P$ does not have order~$3$, we get $A_2\not=0$,
and we scale $X$ and~$Y$ to get $A_2=A_3$.
Then we define $C = 1-A_1$ and $B=-A_2=-A_3$.
This uses up all freedom for changing
Weierstrass equations~\cite[III.3.1(b)]{SilvermanAEC},
so this form is uniquely defined.
The quantity $D$ is the discriminant of~$E$,
which is non-zero.

Conversely, if $D$ is non-zero, then $(E,P)$ defines
an elliptic curve and a point on it, where the point does
not have order $1$, $2$ or $3$.
%sage: P.<B,C> = QQ[]
%sage: E = EllipticCurve([1-C, -B, -B, 0, 0])
%sage: latex(E.discriminant().factor())
%B^{3} \cdot (C^{4} - 8 B C^{2} - 3 C^{3} + 16 B^{2} - 20 B C + 3 C^{2} + B -  C)
\end{proof}

For any elliptic curve $E$ given by a general Weierstrass
equation $y^2 + a_1xy+a_3y = x^3+ a_2x^2+a_4x+a_6$
and any $k\in\ZZ$,
the
\emph{$k$-division polynomial}
$\psi_{k}$
is given by
\begin{align*}
\psi_0 &= 0,
\qquad \psi_2 = 2y+a_1x+a_3,\\
\psi_{k} &= t \cdot 
\!\!\!\!\!\!\!\!\!\!\!\!
 \prod_{Q\in (E[k]\setminus E[2])/\pm}
\!\!\!\!\!\!\!\!\!\!\!\! (x-x(Q)),\qquad\mbox{where}\quad t = \left\{
\begin{array}{ll} k &  \mbox{if $2\nmid k$,}\\ \frac{k}{2}\cdot \psi_2 &\mbox{if $2\mid k$.}
\end{array}\right.
\end{align*}
For any point $P$ on $E$, we have
$k P=0$ if and only if $\psi_{k}(P)=0$.

Let $P_{k}\in\ZZ[B,C]$ be the $k$-division polynomial
$\psi_{k}$ of the elliptic curve \eqref{eq:tateform}
evaluated in the point $P=(0,0)$.
In particular, if $k\geq 4$ and $(E,P)$ corresponds
to $(B,C)\in K^2$ with $D\not=0$, then $P$ has
order dividing $k$ if and only if $P_{k}(B,C)=0$.
%$P_{k} = \psi_{k}(1-C,-B,-B,0,0,0,0)\in\ZZ[B,C]$,
\begin{example}\label{ex:P}
For positive integers~$k$, we compute the $k$-division polynomial
with the SageMath~\cite{sage} command
\[
\href{https://doc.sagemath.org/html/en/reference/arithmetic_curves/sage/schemes/elliptic_curves/ell_generic.html\#sage.schemes.elliptic_curves.ell_generic.EllipticCurve_generic.division_polynomial}{\texttt{E.division\_polynomial(k, two\_torsion\_multiplicity=1)}}\]
and obtain the following list.
\begin{align*}
%sage: P.<B,C> = QQ[]
%sage: E = EllipticCurve([1-C,-B,-B,0,0])
%sage: latex(E.discriminant().factor())
%sage: for k in [1,2,3,4,5,6,7,8]: print "P_%s &= %s \\\\" % (k, latex(E.division_polynomial(k, two_torsion_multiplicity=1)(0,0).factor()))
P_1 &= 1 & P_5 &= -(C - B) \cdot B^{8}\\
P_2 &= - B & P_6 &= - B^{12} \cdot (C^{2} -  B + C)\\
P_3 &= - B^{3} & P_7 &= B^{16} \cdot (C^{3} -  B^{2} + B C)\\
P_4 &= C \cdot B^{5} & P_8 &= C \cdot B^{21} \cdot (B C^{2} - 2 B^{2} + 3 B C -  C^{2})
\end{align*}
\end{example}

For $k\geq 4$, let $F_{k}\in\QQ[B,C]$ be $P_{k}$ with all factors
in common with $D$ and $P_{d}$ for $d<k$ removed
(well-defined up to $\QQ^*$).
Following~\cite{derickx-vanhoeij}, we let
$F_3 = B\in\ZZ[B,C]$ and $F_2 = B^4/D\in\QQ(B,C)$.
\begin{example}
\begin{align*}
F_2 &= B \cdot (C^{4} - 8 B C^{2} - 3 C^{3} + 16 B^{2} - 20 B C + 3 C^{2} + B -  C)^{-1}\\
F_3 &= B\\
F_4 &= C\\
F_5 &= C - B \\
F_6 &= C^{2} -  B + C \\
F_7 &= C^{3} -  B^{2} + B C \\
F_8 &= B C^{2} - 2 B^{2} + 3 B C -  C^{2}
\end{align*}
\end{example}

For $N\geq 4$, the point $P=(0,0)$ on $E$
is of order $N$
if and only if $F_N=0$.
In particular,
we get the following known model of $Y^1(N)$.
\begin{proposition}\label{prop:modularcurve}
	Given $N\geq 4$,
let $R = \QQ[B,C,D^{-1}]\subset \QQ(B,C)$
and let $Y = \mathrm{Spec}(R/F_N)\subset \mathrm{Spec}(R)\subset \AA^2$.
In other words, let $Y$ be the curve over $\QQ$ in the affine $B,C$-plane
given by
\[ Y : F_N = 0, D \not=0.\]
Then for all fields $K$ of characteristic~$0$,
we have $Y^1(N)(K) = Y(K)$.\qed
\end{proposition}
In fact, with a more careful analysis of the Tate normal form
and division polynomials,
one would get the following much stronger result, which we do not need for our main results,
but which we give for completeness.
\begin{proposition}[{Jin~\cite[Corollary 45]{jinbi-divpol}}]\label{prop:jin}
Let $R' = \ZZ[B,C,D^{-1},1/N]\subset R$.
The scheme $\mathrm{Spec}(R'/F_N)$ represents the
``naive'' $\Gamma_1(N)$ moduli problem of \cite{jinbi-divpol}
over $\ZZ[1/N]$.\qed
\end{proposition}

For every $k\geq 2$, the element $F_k\in \QQ[B,C]$ now coincides
with 
$F_k$ of Derickx and Van Hoeij~\cite{derickx-vanhoeij}.
It is irreducible in $\overline{\QQ}[B,C]$ for $k\geq 4$
because the curve $Y^1(k)_{\CC}$ is irreducible.

By taking $B$, $C$, $D$, $F_k$
and $P_k$ modulo $F_N$, we get modular functions
$b$, $c$, $d$, $f_k$ and $p_k$ on $Y^1(N)$
for all $k, N\in\ZZ$ with $k\geq 2$,  $N\geq 4$, and $N\nmid k$.
Derickx and Van Hoeij show (\cite[Section~2]{derickx-vanhoeij}) that they are
\emph{modular units}, that is, functions
with divisors supported at the cusps.
Let $\mathcal{O}(Y^1(N))^*\subset \QQ(X^1(N))^*$
be the group of all modular units.
Our main result is the following.
\begin{theorem}[rephrasing of Theorem~\ref{thm:mainthm} above, Conjecture 1 of \cite{derickx-vanhoeij}]
\label{thm:conjdvh}
The group $\mathcal{O}(Y^1(N))^*/\QQ^*$ is the free
abelian group on
$f_2,f_3,f_4,\ldots,f_{\lfloor N/2\rfloor +1}$.
\end{theorem}
The first, small step of the proof is to rewrite the theorem
in terms of $p_{k}$ using the following lemma.
\begin{lemma}\label{lem:equivalent}
For all $k\geq 3$,
we have $\langle F_2,F_3,\ldots,F_{k}\rangle\cdot \QQ^*
= \langle B, D, P_4, P_5,\ldots, P_{k}\rangle\cdot\QQ^*$.
\end{lemma}
\begin{proof}
Let $L_{k}$ be the left hand side and $R_{k}$ the right.
We prove by induction on~$k$ that we have $L_{k}=R_{k}$
and that all irreducible factors of both $D$ and $P_d$ for $d\leq k$
are elements of~$L_{k}$.

For $k=3$, we have $F_3 = B$ and $F_2=B^4/D$
by definition, hence also $D=F_3^4/F_2^{\phantom{4}}$.
As $B$ and $D/B^3 = F_2^{-1}F_3^{\vphantom{-1}}$
are irreducible, the induction hypothesis follows for $k=3$.

Suppose now that the induction hypothesis holds for $k=n-1$.
By definition $F_n$ is $P_n$ except for factors in common with
$D$ and 
$P_d$ for $d<k$, but by the induction
hypothesis all such factors are in $L_{n-1}=R_{n-1}$.
In particular, we get $L_{n}=R_{n}$.
The polynomial $F_n$ is irreducible as mentioned below Proposition~\ref{prop:jin}, hence
the induction hypothesis also holds for $k=n$.
\end{proof}
By Lemma~\ref{lem:equivalent}, we find that
Theorem~\ref{thm:conjdvh}
is equivalent
to the following.

\begin{theorem}\label{thm:mainthmthree}
The group $\mathcal{O}(Y^1(N))^*/\QQ^*$ is 
the free abelian group on
$b$, $d$, $p_4$, $p_5$, \ldots, $p_{\lfloor N/2\rfloor +1}$.
\end{theorem}

\begin{remark}\label{rem:recurrence}
The division polynomials~$\psi_{k}$, and hence the polynomials
$P_{k}$ and the functions~$p_{k}$, satisfy the following
recurrence relation. For all $m,n,k\in\ZZ$, we have
\begin{align*}
 \psi_{m+n}\psi_{m-n}\psi_{k}^2
&= \psi_{m+k}\psi_{m-k}\psi_{n}^2
-\psi_{n+k}\psi_{n-k}\psi_{m}^2.
\intertext{
Taking $(k,m,n) = (1,l+1,l)$ or $(1,l+1,l-1)$, 
we get}
  \psi_{2l+1}^{\vphantom{1}} &= \psi_{l+2}^{\vphantom{1}}\psi_{l}^3 - \psi_{l+1}^3\psi_{l-1}^{\vphantom{1}},\\
\psi_{2l}^{\vphantom{1}} &= \psi_{2}^{-1} \psi_{l}^{\vphantom{1}} \left(\psi_{l+2}\psi_{l-1}^2
-\psi_{l-2}^{\vphantom{1}}\psi_{l+1}^2
\right),
\end{align*}
which gives~$p_{k}$ for all $k \geq 5$ starting
from the initial terms
$p_1$, $p_2$, $p_3$, $p_4$
of Example~\ref{ex:P}.
\end{remark}

\newcommand{\computercode}[1]{}
\begin{example}
The curve $X^1(5)$ is defined by
$0 = F_5 = C-B$, that is, by $B=C$.
We compute
\computercode{
P.<c> = QQ[]
C = c
B = c
E = EllipticCurve([1-C, -B, -B, 0, 0])
def myfactor(k):
	if k == 0:
		return 0
	return k.factor()

for k in [1,2,3,4,5,6,7,8,9,10]:
    print "p_{" + str(k) + "} &=& " + latex((E.division_polynomial(k, two_torsion_multiplicity=1)(0,0)(c))) + "\\\\"
   
latex(factor(E.discriminant()))
--
}
\[
\begin{array}{lcllcl}
p_{1} &=& \phantom{-}1 & p_{6} &=& -c^{14}\\
p_{2} &=& -c & p_{7} &=& \phantom{-}c^{19}\\
p_{3} &=& -c^{3} & p_{8} &=& \phantom{-}c^{25}\\
p_{4} &=& \phantom{-}c^{6} & p_{9} &=& -c^{32}\\
p_{5} &=& \phantom{-}0 & p_{10} &=& \phantom{-}0\\
d &=& \phantom{-}c^{5} \cdot (c^{2} - 11 c - 1),
\end{array}
\]
which, except for $p_5$ and $p_{10}$,
all lie in the group generated by $b=c$ and~$d$.
\end{example}

\begin{example}
The curve $X^1(6)$ is defined by
$0=F_6=C^2-B+C$, that is, by
$B = C(C+1)$.
We compute
\computercode{
P.<c> = QQ[]
C = c
B = c*(c+1)
E = EllipticCurve([1-C, -B, -B, 0, 0])
def myfactor(k):
	if k == 0:
		return 0
	return k.factor()

for k in [1,2,3,4,5,6,7,8,9,10]:
    print "p_{" + str(k) + "} &=& " + latex(myfactor(E.division_polynomial(k, two_torsion_multiplicity=1)(0,0)(c))) + "\\\\"
   
print "d &=& " + latex(factor(E.discriminant()))
--
}
\[
\begin{array}{lcllcl}
p_{1} &=& \phantom{-}1 & p_{6} &=& \phantom{-}0\\
p_{2} &=& -  c^{\phantom{60}} \cdot (c + 1) & p_{7} &=& -  c^{20}\cdot (c + 1)^{16}\\
p_{3} &=& -  c^{3\phantom{0}} \cdot (c + 1)^{3} & p_{8} &=& - c^{26}\cdot (c + 1)^{21}\\
p_{4} &=& \phantom{-} c^{6\phantom{0}} \cdot (c + 1)^{5} & p_{9} &=& \phantom{-}c^{33}\cdot (c + 1)^{27}\\
p_{5} &=& \phantom{-}  c^{10} \cdot (c + 1)^{8} & p_{10} &=& \phantom{-}c^{41} \cdot (c + 1)^{33}\\
d &=& \phantom{-} c^{6\phantom{0}}\cdot (c + 1)^{3}\cdot (9c + 1),
\end{array}
\]
which indeed all, except for $p_6$,
lie in the group generated by $b=c(c+1)$, $d$ and $p_4$.
\end{example}

\subsection{Siegel functions}\label{sec:siegelfunctions}

This section defines the \emph{Siegel functions}
of Theorem~\ref{thm:mainthmtwo} and recalls their
transformation properties and $q$-expansions.
Our main reference for this section is Fricke~\cite{fricke}.
We start by recalling the well-known Weierstrass sigma
function and Dedekind eta function.

\subsubsection{Lattices, sigma and eta}

By a \emph{lattice}, we will always mean a discrete subgroup
$\Lambda\subset \CC$ of rank~2.
For example, for $\tau\in\HH$, we have a lattice
$\Lambda_\tau = \tau\ZZ+\ZZ$.
For $\omega_1,\omega_2\in\CC$
with $\tau=\omega_1/\omega_2 \in\HH$,
we have a lattice $\omega_1\ZZ+\omega_2\ZZ = \omega_2 \Lambda_\tau$.

We define the \emph{Weierstrass sigma function}
by
(\cite[(1) on p.258]{fricke})
\[ \sigma(z,\Lambda)
= z \prod_{
\substack{w\in\Lambda\\
\omega\not=0}}
(1-\frac{z}{w})\exp
\left(\frac{z}{w}+\frac{1}{2} (\frac{z}{w})^2\right)\]
for all $z\in\CC$ and all lattices
$\Lambda\subset\CC$.
We also define $\sigma(z,\tau)=\sigma(z, \Lambda_\tau)$.

Let $\zeta(z,\Lambda) = \frac{\frac{d}{dz}\sigma(z,\Lambda)}{\sigma(z,\Lambda)}$
be the logarithmic derivative of~$\sigma$
(\cite[(6) on p.209]{fricke}).
It is quasi-periodic in the sense that
we have
$$\zeta(z+\omega_i,\Lambda) = \zeta(z,\Lambda) + \eta_i,$$
for some $\eta_1,\eta_2\in\CC$, which we
call the \emph{basic quasi periods}
associated to $\omega_1,\omega_2$~\cite[(4) on p.196]{fricke}.
They satisfy the Legendre relation
$\omega_1\eta_2-\omega_2\eta_1 = 2\pi i$
(\cite[(6) on p.160]{fricke}).

Let $\eta$ (not to be confused with $\eta_1$ and~$\eta_2$)
be the \emph{Dedekind eta function}
$$\eta(\tau) = q^{1/24}\prod_{n=1}^{\infty} (1-q^n)
\qquad\mbox{where}\qquad
q=\exp(2\pi i \tau)
.$$

\subsubsection{Klein forms and Siegel functions}

For $a=(a_1,a_2)\in\QQ^2\setminus\ZZ^2$,
we define the \emph{Klein form} $\mathfrak{t}_a$
as a function of $\RR$-linearly
independent pairs $\omega_1,\omega_2\in\CC$
by
$$ \mathfrak{t}_a(\omega_1,\omega_2) = \exp \left(-\textstyle{\frac{1}{2}}(a_1\eta_1+a_2\eta_2)(a_1\omega_1+a_2\omega_2)\right)
\sigma\left(a_1\omega_1+a_2\omega_2, \omega_1\ZZ+\omega_2\ZZ\right).$$
There are many variants of the notation for Klein forms
in the literature. Our Klein form $\mathfrak{t}_a$ equals
$-\sigma_{gh}$
in the notation of~\cite[(6) on p.451]{fricke} where $(g/N, h/N)=a$.

Define for $a=(a_1,a_2)\in\QQ^2\setminus \ZZ^2$
the function $\mathfrak{t}_a:\HH\rightarrow \CC$
by
\begin{equation}\label{eq:frakt}
\mathfrak{t}_a(\tau) = \omega_2^{-1}\mathfrak{t}_a(\omega_1, \omega_2),
\end{equation}
for any $\omega_1,\omega_2\in\CC$ with $\omega_1/\omega_2 = \tau$.
Indeed, by \cite[(7) on p.452]{fricke}, this
depends only on $a$ and $\tau$, not on $\omega_1$
and~$\omega_2$.
Our $\mathfrak{t}_a(\tau)$ is exactly $\mathfrak{t}_a({\tau\atop 1})$ of Kubert and Lang~\cite[\S 2.1, p.27]{kubert-lang-book}.

Define
the \emph{Siegel function}
$$ h_a = 2\pi\eta^2 \mathfrak{t}_a,$$
which is $-i$ times the function $g_a$ of Kubert and Lang~\cite[\S 2.1, p.29]{kubert-lang-book}.

\begin{remark}
Our Klein forms and Siegel functions
are the same as those
in Kubert and Lang~\cite{kubert-lang-II,kubert-lang-IV}
up to multiplication by a constant and taking
fractional powers.
Kubert and Lang do not have the factor~$\frac{1}{2}$
in the exponent in the definition
of $\mathfrak{t}_{a}(\omega_1,\omega_2)$~(\cite[p.176]{kubert-lang-II}),
but this is either due to a typo in~\cite{kubert-lang-II}
or
due to different scaling conventions
on e.g.\ $\omega_i$ and/or~$\eta_i$.
Indeed, the definition as we have given it satisfies
\cite[K2 on p.177]{kubert-lang-II},
and it would not have done so without the factor~$\frac{1}{2}$.

The notation of Kubert and Lang varies a bit from paper
to paper. 
For details of the relations between the functions,
see the following equalities, where a superscript
II refers to \cite{kubert-lang-II} and
IV to \cite{kubert-lang-IV}.
Moreover, in the case of II, a positive integer~$N$
is understood to be fixed and we have
$a = (r/N, s/N)$.
Up to constant factors, we have
\[
\begin{array}{rclclcl}
\mathfrak{t}_a
& =& \mathfrak{t}_{r,s}^{\mathrm{II}}
&= &\mathfrak{t}_a^{\mathrm{IV}},
\\
{h}_a
& = & (g_{r,s}^{\mathrm{II}})^{1/(12N)}
&= & h_a^{\mathrm{IV}}
& = & \left\{\begin{array}{ll} g_a^{\mathrm{IV}} & \mbox{if $2a\not\in\ZZ^2$,}\\
 (g_a^{\mathrm{IV}})^2 & \mbox{if $2a\in\ZZ^2$.}
\end{array}\right.
\end{array}\]
\end{remark}

\begin{lemma}\label{lem:propertiesofh}
The Siegel functions $h_a$ have the following expansions and
transformation
properties for all $a=(a_1,a_2)\in\QQ^2\setminus\ZZ$.
\begin{enumerate}
\item\label{item:propertiesofh1}
Write $q^a = \exp(2\pi i a_2)\cdot q^{a_1}=\exp(2\pi i (a_1\tau+a_2))$.
If $0\leq a_1\leq \frac{1}{2}$, then we have
\begin{equation}
\label{eq:q}
h_a = c(a)q^{\frac{1}{2} (a_1^2 - a_1 + \frac{1}{6})}
(1-q^a)\prod_{n=1}^\infty (1-q^nq^a)(1-q^nq^{-a}),
\end{equation}
where $c(a) = i\exp(\pi i a_2(a_1-1))$ is a constant.
\item\label{item:propertiesofh2} $h_{-a} = -h_{a}$.
\item\label{item:propertiesofh3} $h_{(a_1+n_1, a_2+n_2)} = (-1)^{n_1n_2+n_1+n_2}
e^{-\pi i (n_1a_2-n_2a_1)} h_{(a_1,a_2)}$ for all $(n_1,n_2)\in\ZZ^2$,
\item\label{item:propertiesofh4}
$h_{(a_1+1,0)} = -h_{(a_1,0)}$.
\item\label{item:propertiesofh5} $h_{a}$ up to multiplication by roots of unity depends only
on the class of $a$ in $(\QQ^2/\ZZ^2)/\{\pm 1\}$.
\item\label{item:propertiesofh6}
For all $$M = \left(\begin{array}{cc} \alpha & \beta \\
\gamma & \delta\end{array}\right)\in\mathrm{SL}_2(\ZZ),$$
we have 
\begin{equation}\label{eq:htransformation}
h_a(M\tau)
= \epsilon(M) h_{aM}(\tau),
\end{equation}
where $\epsilon(M)\in\CC^*$ is such that
for all $\tau\in\HH$,
\begin{equation}\label{eq:etatransformation}
\eta(M\tau)^2 = \epsilon(M)(\gamma\tau+\delta)\eta(\tau)^2.
\end{equation}
\item\label{item:propertiesofh7} The function $\epsilon$ from \eqref{eq:etatransformation} satisfies
$\epsilon(M)^{12} = 1$ and
$$\epsilon\left(\left(\begin{array}{rr} 1 & 0 \\ 1 & 1\end{array}\right)\right) = \exp(2\pi i / 12)^{-1}.$$
\end{enumerate}
\end{lemma}
\begin{proof}
The expansion in \eqref{item:propertiesofh1} is Fricke~\cite[(7) on p.452]{fricke},
but note that our $q$ is the square of the $q$ of Fricke.
% TO KEEP IN COMMENTS:
% Recall t_a = - sigma_{gh}
% Now sigma_{gh} eta^2 / q^{1/12}
%      =   omega_2 / (2*i*pi)
%        * exp(pi*i*h(g-n)/n^2) q^(g(g-n)/n^2)
%        * (explicit q-expansion)
% So h_a = 2*pi*eta^2 * omega_2^{-1} t_a
%        = - 2*pi*eta^2 omega_2^{-1} *sigma_{gh}
%        = -2*pi*q^{1/12} omega_2^{-1} * LHS
%        = -2*pi*q^{1/12} omega_2^{-1} * RHS
%        = -1/i * exp(pi*i*h(g-n)/n^2) q^(1/12 + g(g-n)/n^2)
%          * (explicit q-expansion)
%  Here -1/i = i, so c(a) = i exp(pi*i*h(g-n)/n^2)
Equivalently, the expansion is $-i$ times
Kubert and Lang's (\cite[K5 on p.178]{kubert-lang-II}
or equivalently \cite[K4 on p.29]{kubert-lang-book}).
%also~\cite[p.228]{kubert-lang-IV}, (up to constants)
% TO KEEP IN COMMENTS:
% their $\Delta^{1/12}$ is i times our $(2\pi \eta^2)$,
% so if we multiply our equality by i, then we get exactly K5.

The identity $h_{-a} = -h_{a}$ of \eqref{item:propertiesofh2} follows from the anti-symmetry of $\sigma$
as a function of~$z$.

The identity of \eqref{item:propertiesofh3} is Fricke~\cite[(4) on p.451]{fricke}.
% TO KEEP IN COMMENTS:
% t_a = -sigma_{gh} = -sigma_{gh}(0), so take u=0 in Fricke (4)
Identity~\eqref{item:propertiesofh4} is a special case of~\eqref{item:propertiesofh3}.

Observation \eqref{item:propertiesofh5} follows immediately from \eqref{item:propertiesofh2} and~\eqref{item:propertiesofh3}.

As $\eta^{24}$ has level~$1$, if we let $\epsilon(M,\tau) = \eta(M\tau)^2 / ((\gamma\tau+\delta)\eta(\tau)^2)$, then we get $\epsilon(M,\tau)^{12}=1$, hence
$\epsilon(M,\tau)$ is independent of~$\tau$, call it~$\epsilon(M)$.
A numerical evaluation yields the example value of~\eqref{item:propertiesofh7},
so it remains to prove equality~\eqref{eq:htransformation}
in~\eqref{item:propertiesofh6}.

First, \cite[(3) on p.451]{fricke}
(equivalently \cite[K1 on p.177]{kubert-lang-II})
gives 
$$\mathfrak{t}_a(M\left({\omega_1 \atop \omega_2}\right))
= \mathfrak{t}_{aM}(\left({\omega_1 \atop \omega_2}\right)).$$
In terms of $\tau=\omega_1/\omega_2$, this reads (by \eqref{eq:frakt})
$$(\gamma\omega_1+\delta\omega_2)\mathfrak{t}_a(M\tau)
= \omega_2 \mathfrak{t}_{aM}(\tau).$$
Now multiply this equality by
$2\pi$ and~\eqref{eq:etatransformation} to get
$$(\gamma\omega_1+\delta\omega_2)h_a(M\tau)
 = (\gamma\tau+\delta)\omega_2\epsilon(M) h_{aM}(\tau),$$
which proves~\eqref{eq:htransformation}.
\end{proof}

We use the shorthand notation
\begin{equation}\label{eq:Hk2} H_k = h_{(k/N,0)},\end{equation}
which by Lemma~\ref{lem:propertiesofh}\eqref{item:propertiesofh1}
is the same as~\eqref{eq:Hk}.

\subsection{\texorpdfstring{Remarks on the difference between $\Gamma^1$ and $\Gamma_1$}{Remarks on the difference between Gamma1 and Gamma1}}

The curve that we denote by $Y^1(N)$ is often denoted $Y_1(N)$,
mostly by authors who prefer to use
the group $\Gamma_1(N)$ instead, which is defined
as in \eqref{eq:gamma1n}
with $c\equiv 0$ instead of $b\equiv 0$.
We now give two remarks for how to adapt Theorem~\ref{thm:mainthmtwo}
to that situation. We will not use these remarks in the rest of this article.

\begin{remark}\label{rem:altparam2}
	There is a complex analytic isomorphism
	\begin{align}\label{eq:altparam}
		\Gamma_1(N)\backslash \HH &\longrightarrow  Y^1(N)(\CC) \\
		\tau &\longmapsto (\CC / \Lambda_\tau, 1/N\ \mathrm{mod}\ \Lambda_\tau).\nonumber
	\end{align}
	The field of functions on $X^1(N)$ defined over $\QQ$
	with that choice of parametrization is
	the field
	of meromorphic functions on $\Gamma_1(N)\backslash\HH^*$
	whose expansion at the cusp $0$ is rational,
	that is, the functions
	in $\QQ((\exp(-2\pi i \tau^{-1})))$.
	
	The isomorphism $\Gamma_1(N)\backslash \HH\rightarrow \Gamma^1(N)\backslash \HH$
	obtained by composing the two par\-am\-et\-riz\-ations \eqref{eq:altparam} and~\eqref{eq:ourparam} is given by
	$$\begin{psmallmatrix}\phantom{-}0 & \phantom{-}1 \\ -1 & \phantom{-}0\end{psmallmatrix}:\tau\mapsto -\tau^{-1}.$$
	In particular, if one uses the
	parametrization \eqref{eq:altparam},
	then in
	Theorem~\ref{thm:mainthmtwo} one should
	replace $H_{k}(\tau)$ by $H_{k}(-1/\tau)$.
\end{remark}

\begin{remark}\label{rem:nonstandard}
	There is another
	complex analytic isomorphism,
	given by
	\begin{align}\label{eq:nonstandard}
		\Gamma_1(N)\backslash \HH &\longrightarrow  Y^1(N)(\CC) \\
		\tau &\longmapsto (\CC / \Lambda_{N\tau}, \tau\ \mathrm{mod}\ \Lambda_{N\tau}).
		\nonumber
	\end{align}
	The field of functions on $X^1(N)$ defined over $\QQ$
	with that choice of parametrization is the field
	of meromorphic functions on $\Gamma_1(N)\backslash\HH^*$
	whose expansion at the cusp $\infty$ is rational,
	that is, the function is
	in $\QQ((\exp(2\pi i \tau)))$.
	
	The isomorphism $\Gamma_1(N)\backslash \HH\rightarrow \Gamma^1(N)\backslash \HH$
	obtained by composing the two par\-am\-et\-riz\-ations
	\eqref{eq:nonstandard} and~\eqref{eq:ourparam} is given by
	$$\begin{psmallmatrix}N & 0 \\ 0 & 1\end{psmallmatrix}:\tau\mapsto N\tau.$$
	In particular, if one uses 
	the parametrization \eqref{eq:altparam},
	then in
	Theorem~\ref{thm:mainthmtwo} one should
	replace $H_{k}(\tau)$ by $H_{k}(N\tau)$,
	which is denoted by $iE_{k}(\tau)$ in \cite{yifanyang,yifanyang2}.
\end{remark}

From now on, we only use the parametrization~\eqref{eq:ourparam},
and will not use Remark
\ref{rem:altparam2} or~\ref{rem:nonstandard}.

\section{Relating the functions}
\label{sec:relating}

We now give the first part of the proof
of the main theorems: relating the groups 
given by the sets of generators
of the theorems.
We start by expressing the functions $P_{k}$
and $p_{k}$ of Section~\ref{sec:tateform}
in terms of the Weierstrass $\sigma$-function.

\subsection{The Weierstrass sigma function}
\label{sec:sigma}

To any lattice $\Lambda\subset\CC$ of rank two
and any $z\in\CC$, we associate
an elliptic curve $E$ with $E(\CC)=\CC/\Lambda$
and a point $P=(z\ \mathrm{mod}\ \Lambda)$.

The curve $E$ has a 
classical Weierstrass equation
\begin{equation}
\label{eq:classical}
W: y^2 = 4x^3 - g_2(\Lambda) x - g_3(\Lambda),
\end{equation}
where $g_2(\Lambda) = 60\sum_{\omega\in\Lambda\setminus\{0\}} \omega^{-4}$
and $g_3(\Lambda) = 140\sum_{\omega\in\Lambda\setminus\{0\}} \omega^{-6}$.
We let $\Delta = \Delta(\Lambda) = 16(g_2(\Lambda)^3-27g_3(\Lambda)^2)$
be the discriminant of the right hand side of \eqref{eq:classical}.

After putting the pair $(E,P)$ in Tate normal form,
we get $B$ and $C$ as functions in $z$ and~$\Lambda$.
In particular, we get expressions for $P_{k}$ in terms of $z$ and~$\Lambda$.
The following result gives these expressions.

\begin{proposition}\label{prop:sigma}
For any positive integer~$k$, let
\[
\Phi_{k} = \frac{\sigma(kz, \Lambda)}{\sigma(z,\Lambda)^{k^2}}\quad\mbox{and}\quad
U =\frac{\Phi_3}{\Phi_2^3}.
\]
Then we have
\[ P_{k} = U^{k^2-1}\Phi_{k}
\quad\mbox{and}\quad D = U^{12} \Delta.\]
\end{proposition}
\begin{proof}
Let $\wp(z,\Lambda)$
be the Weierstrass $\wp$-function
and $\wp' = \frac{d}{dz} \wp$.
Then for any $v\in\CC$,
we get a point
$(x,y) = (\wp(v,\Lambda), \wp'(v,\Lambda))$
on~\eqref{eq:classical}.

We put the classical Weierstrass equation
$W$ in Tate normal form relative
to the point $P=(x_0,y_0) = (\wp(z,\tau), \wp'(z,\tau))$.
The transformation is of the form
$X = u^2 (x+t)$, $Y=\frac{1}{2}u^3(y+rx+s)$
with $u,r,s,t$ functions of $z$ and $\tau$,
where $X$ and $Y$ are the coordinate functions
for the Tate normal form and $x$ and $y$ are the coordinate
functions for the classical Weierstrass equation.

First, we compute the discriminant $D$ of the Tate normal form.
Completing the square to get an equation of the form ${Y'}^2=X^3+\cdots$
does not affect the discriminant or the $X$-coordinates of the
two-torsion points. 
Note that the discriminant of a
Weierstrass equation $(Y')^2 = X^3 + \cdots$
is $16$ times the discriminant of the right hand side.
Let $Q_1$, $Q_2$, $Q_3$ be the points of order $2$ on $E$.
Then
\begin{align*}
D &= 16\cdot (X(Q_1)-X(Q_2))^2\cdot (X(Q_2)-X(Q_3))^2\cdot (X(Q_3)-X(Q_1))^2\\
&= 16u^{12}\cdot (x(Q_1)-x(Q_2))^2\cdot (x(Q_2)-x(Q_3))^2\cdot (x(Q_3)-x(Q_1))^2 \\
&= u^{12}\Delta.
\end{align*}
Similarly, we have
\begin{equation}\label{eq:Pksqrt1}
P_{k} = k \sqrt{\prod_{Q\in E[k]\setminus\{0\}} (X-X(Q))}
       = u^{(k^2-1)} k \sqrt{\prod_{Q\in E[k]\setminus\{0\}} (x-x(Q))},
\end{equation}
where the square roots are is chosen to
be monic polynomials
times~$1$ or times $Y+\frac{1}{2}a_1X+\frac{1}{2}a_3$.

We use the classical identity
\begin{equation}\label{eq:Pksqrt2}
(-1)^{k+1}\ k\ \sqrt{\prod_{Q\in E[k]\setminus\{0\}} (x-x(Q))} = 
\frac{\sigma(k z, \Lambda) }{ \sigma(z,\Lambda)^{k^2}}.
\end{equation}
For a proof, see
Theorem~2.7 of De~Looij~\cite{looij-thesis}.
The factor $(-1)^{k+1}$ does not appear in~\cite{looij-thesis},
but our choice of square root differs from the choice in loc.~cit.~by
exactly that factor.
The proof in \cite{looij-thesis}
works by fixing the lattice $\Lambda$ and showing that both sides
are elliptic functions for that lattice
with the same divisor and with
equal leading terms in their power series.

Combining \eqref{eq:Pksqrt1} and \eqref{eq:Pksqrt2}, we get
\begin{equation}
P_{k} = (-u)^{k^2-1} \frac{\sigma(kz, \Lambda) }{ \sigma(z,\Lambda)^{k^2}}
= (-u)^{k^2-1} \Phi_{k},
\end{equation}
so it suffices to prove $-u = U$.

Proving $-u= U$ could be done by a lengthy computation of the Tate normal
form from~$W$. Instead, simply note
$$1 = \frac{B^3}{B^3}
= \frac{P_3}{P_2^3} = \frac{(-u)^{3^2-1}}{(-u)^{3(2^2-1)}}\frac{\Phi_3}{\Phi_2^3} = (-u)^{-1}U,$$
which finishes the proof.\end{proof}

Next, we specialize to $\Lambda=\Lambda_\tau$ and $z=\tau/N$
consistently with the identification
$\Gamma^1(N)\backslash\HH\rightarrow Y^1(N)(\CC)$
of~\eqref{eq:ourparam}.

\begin{corollary}
\label{cor:sigma}
For any integer $N\geq 4$ and any positive integer $k$ with $N\nmid k$, let
\[\phi_{k} = \frac{\sigma\!\left(\frac{k\tau}{N}, \tau\right)}{\sigma\!\left(\frac{\tau}{N},\tau\right)^{k^2}}\quad\mbox{and}\quad
u = \frac{\phi_3}{\phi_2^3}.\]
Then the following identities of meromorphic
functions hold
on $X^1(N)$:
\[
p_{k} = u^{k^2-1}\phi_{k}
\quad\mbox{and}\quad d =
(2\pi \eta^2 u)^{12}.\]
\end{corollary}
\begin{proof}
Take $\Lambda=\Lambda_\tau$ and
$z=\tau/N$ in Proposition~\ref{prop:sigma},
and
use the well known equality
$\Delta(\Lambda_\tau) = (2\pi \eta(\tau)^2)^{12}$
(see \cite[(6) on p.~313]{fricke}).
\end{proof}

\subsection{\texorpdfstring{The functions $p_{k}$ in terms of the functions $H_k$}{The functions p in terms of the functions H}}
\label{sec:intermsof1}

Now that we have expressed the functions $p_{k}$
in terms of
Weierstrass~$\sigma$-functions, we use these expressions
to express the~$p_{k}$ in terms of Siegel functions.
\begin{lemma}\label{lem:express}
Let
$$ t = \frac{H_{1}^2H_{3}}{H_{2}^3}.
$$
Then for all integers $N\geq 4$ and $k\in\ZZ\setminus N\ZZ$ we have 
\[ p_{k} = t^{k^2-1} \frac{H_{k}}{H_{1}}
\quad
\mbox{and}
\quad
d = (tH_{1})^{12}.
\]
\end{lemma}
\begin{proof}
In the notation of Corollary~\ref{cor:sigma},
we have
\begin{align*}
\phi_{k} &= \frac{\sigma(k\tau/N, \tau)}{\sigma(\tau/N,\tau)^{k^2}}
= \frac{\mathfrak{t}_{(k/N,0)}}{\mathfrak{t}_{(1/N,0)}^{k^2}}
= \frac{H_{k}}{H_{1}^{k^2}}(2\pi \eta^{2})^{k^2-1}\\
&= \frac{H_{k}}{H_{1}}\left(\frac{H_{1}}{2\pi \eta^2}\right)^{1-k^2},
\\
u &= \phi_3\phi_2^{-3}
 = t\cdot \left(\frac{H_{1}}{2\pi \eta^2}\right),\\
p_{k} &= u^{k^2-1}\phi_{k} = t^{k^2-1} \frac{H_{k}}{H_{1}},\\
d &= \left(2\pi \eta^2 u\right)^{12}
= (tH_{1})^{12},
\end{align*}
so the result follows.
\end{proof}
Let $m = \lfloor N/2\rfloor$.
Next, we express $p_{m+1}$ in terms of $H_{k}$ with
$1\leq k\leq m$ using the periodicity and symmetry of $H_{k}$ in~$k$.
\begin{lemma}\label{lem:express2}
Let $t$
be as in Lemma~\ref{lem:express}, let
$m = \lfloor N/2\rfloor$, and let $v = t^{\gcd(2,N)N}$.
Then we have 
\begin{align*}
p_{m+1} &= \left\{ \begin{array}{ll}
   v p_{m},\quad & \mbox{if $N$ is odd},\\
   v p_{m-1},\quad & \mbox{if $N$ is even}.
\end{array}\right.
\end{align*}
Moreover, each of $d, p_2, p_4,p_5,p_6,\ldots,p_{m+1}$
(including $-b=p_2$)
is of the form
$$f = \prod_{k=1}^m H_{k}^{e(k)},$$
where for every $k\in\{1,2,\ldots,m\}$ we have $e(k)\in\ZZ$, and where we have
\begin{equation}\label{eq:modulo}
\sum_{k=1}^m e(k) \in 12\ZZ\qquad\mbox{and}\qquad
\sum_{k=1}^m k^2 e(k) \in N\gcd(N,2)\ZZ.
\end{equation}
\end{lemma}
\begin{proof}
Suppose first that $N$ is odd, so $N = 2m+1$.
Lemma~\ref{lem:express} gives
$$p_{m+1} = t^{(m+1)^2-1} H_{m+1}/H_{1}$$
and by Lemma~\ref{lem:propertiesofh} (parts \eqref{item:propertiesofh2} and~\eqref{item:propertiesofh4}), we have
$H_{m+1} = -H_{-(m+1)} = H_{m}$, hence
$$p_{m+1} = t^{2m+1} t^{m^2-1} H_{m}/H_{1}=v p_m.$$

If $N$ is even, then $N = 2m$ and $t^{(m+1)^2-1} = t^{4m} t^{(m-1)^2-1}$,
so the same calculation gives $p_{m+1} = v p_{m-1}$.

A straightforward calculation verifies~\eqref{eq:modulo}
for each expression in Lemma \ref{lem:express}
or~\ref{lem:express2}.
Indeed, the value of $(\sum_{k} e(k), \sum_k k^2e(k))\in\ZZ^2$
is \begin{align*}
(1, n^2) \qquad&\mbox{for $H_{n}$ for all $n\in\ZZ$ with $N\nmid n$},\\
(0, -1) \qquad &\mbox{for $t$},\\
(12, 0) \qquad&\mbox{for $d$},\\
(0, 0) \qquad&\mbox{for $p_{n}$ with $1 \leq n\leq m$ (Lemma~\ref{lem:express})},\\
(0, -\gcd(N,2)N)\qquad &\mbox{for $v$ and hence for $p_{m+1}$ (Lemma~\ref{lem:express2})}.\qedhere
\end{align*}
\end{proof}

\subsection{\texorpdfstring{The functions $H_k$ in terms of $p_{k}$}{The functions H in terms of p}}
\label{sec:intermsof2}

Now that we have expressions of $p_{k}$ in terms of $H_{k}$,
it is a matter of solving a system of linear
equations to obtain
the reverse expressions.
These expressions are given in the following result.
\begin{proposition}\label{prop:hinp}
Let $m = \lfloor N/2\rfloor$.
Given $e\in\ZZ^m$ satisfying~\eqref{eq:modulo} and
given $$f = \prod_{k=1}^m H_{k}^{e(k)},$$
let $\alpha = \frac{1}{12}\sum_k e(k)$
and
$\beta = (N\gcd(2,N))^{-1}\sum_k k^2 e(k)$.
Then we have
\begin{equation}
\label{eq:hinp}
 f = d^\alpha \left(p_{N-m-1}^{\vphantom{-1}} p_{m+1}^{-1}\right)^{\beta} \prod_{k=1}^{m} p_{k}^{e(k)},
\end{equation}
where $p_1=1$, $p_2 = -b$, $p_3=-b^3$,
and $N-m-1\in\{m-1,m\}$,
so $$f\in \langle -b, d, p_4, p_5, \ldots, p_{m+1}\rangle
\subset \mathcal{O}(Y^1(N))^*.$$
\end{proposition}
\begin{proof}
Note that 
Lemma~\ref{lem:express}
gives
\[\prod_{k=1}^{m} p_{k}^{e(k)}
= t^{\sum_{k} k^2 e(k)^2} (t H_{1})^{-\sum e(k)}\prod_{k=1}^{m} H_{k}^{e(k)}
= v^\beta d^{-\alpha}\prod_{k=1}^{m} H_{k}^{e(k)}
.\]
As Lemma~\ref{lem:express2} gives $v = p_{m+1}^{\phantom{-1}}p_{N-m-1}^{-1}$,
this proves~\eqref{eq:hinp}.
The formulas for $p_1$, $p_2$ and $p_3$ are
in Example~\ref{ex:P}.
\end{proof}
The following result sums up in how far we have now
proven the main theorems.
\begin{proposition}\label{prop:express}
Let $S$ be the group of functions
of the form $\prod_{k=1}^m H_{k}^{e(k)}$
satisfying~\eqref{eq:modulo}.
If $S$ has rank~$m$
and $\QQ^*\cdot S$ contains
$\mathcal{O}(Y^1(N))^*$, then all of
Theorems~\ref{thm:mainthm}, \ref{thm:mainthmtwo},
\ref{thm:conjdvh}
and~\ref{thm:mainthmthree} hold.
\end{proposition}
\begin{proof}
	Let $T = \langle -b, d, p_4,p_5,\ldots, p_{m+1}\rangle\subset \mathcal{O}(Y^1(N))^*$.
	Lemma~\ref{lem:express} and Proposition~\ref{prop:hinp}
	show $S = T$, hence also $\QQ^*\cdot S \subset \mathcal{O}(Y^1(N))^*$.
	
	The leading coefficients of the $q$-expansions of the functions $H_{k}(\tau)$
	are all~$i$ by~\eqref{eq:Hk}, hence the leading coefficients of $q$-expansions of
	the elements of $S$
	are all~$1$, so $S\cap \QQ^* = 1$.
	In particular, the rank of $(\QQ^*\cdot S)/\QQ^*$ equals the rank of~$S$.
	
	Under the assumption that this rank is $m$ and that $\QQ^*\cdot S$ contains
	$\mathcal{O}(Y^1(N))^*$, we get exactly Theorems \ref{thm:mainthmtwo}
	and~\ref{thm:mainthmthree}.
	
	By Lemma~\ref{lem:equivalent}, Theorem~\ref{thm:mainthmthree} implies Theorems \ref{thm:mainthm} and~\ref{thm:conjdvh}.
\end{proof}

\section{\texorpdfstring{$q$-expansions and Gauss' Lemma}{q-expansions and Gauss' Lemma}}
\label{sec:powerseries}

Recall that $S$ is the group of functions
of the form $\prod_{k=1}^m H_{k}^{e(k)}$
satisfying~\eqref{eq:modulo},
where $m=\lfloor N/2\rfloor$.
As stated in Proposition~\ref{prop:express},
it now suffices to prove that $S$ has rank~$m$
and $\mathcal{O}(Y^1(N))^*\subset \QQ^*\cdot S$.

Section~\ref{sec:rank} uses $q$-expansions to show that
the Siegel functions $H_{k}$ for $k=1,2,\ldots, m$
are multiplicatively independent. The group they generate then has
the correct rank.

Section~\ref{sec:gauss} combines this with
Gauss' Lemma for power series
to show that $\mathcal{O}(Y^1(N))^*$
is contained
in $\QQ^*\cdot \langle H_1,H_2,\ldots, H_{m-2},H_{m-1}, H_m^{1/2}\rangle$.

Section~\ref{sec:proof}
then uses explicit $\SL_2$-actions to find restrictions on the exponent vectors,
finishing the proof of $\mathcal{O}(Y^1(N))^*\subset \QQ^*\cdot S$.

\subsection{The rank}
\label{sec:rank}

\begin{proposition}
\label{prop:linearlyindep}\label{prop:defreducedform}
The functions $H_{k}$ for $k=1,2,\ldots,m$
are multiplicatively independent modulo~$\CC^*$.
In other words, if
$$\prod_{k=1}^{m} H_{k}^{e(k)}\in\CC^*$$
with $e\in\ZZ^k$, 
then $e=0$.
\end{proposition}
\begin{proof}
We prove the result
using $q$-expansions.
Following \cite{kubert-lang-IV}, we define the
\emph{reduced form} $f^*$ of a non-zero Laurent series $f$
to be $f$ divided by its lowest-degree term,
so $f^* = 1 + \mbox{higher order terms}$.

From~\eqref{eq:Hk}, we have for $0<k\leq N/2$:
\begin{align}\nonumber
H_{k}^* &= (1-q^{k/N})\prod_{n=1}^\infty (1-q^{n+k/N})(1-q^{n-k/N})\\
&= \left\{ \begin{array}{ll}
1 - q^{k/N} + O(q^{1-k/N})&\qquad\mbox{if $0<k<N/2$, and}\\
 1 - 2q^{1/2} + O(q^{3/2})&\qquad\mbox{if $k=N/2$}.\end{array}\right.
 \label{eq:qreduced}
\end{align}

Suppose that we have
$\prod_{k=1}^m H_{k}^{e(k)}\in\CC^*$
for some $0\not=e\in\ZZ^m$.
Let $k_0$ be the smallest positive
integer with $e(k_0)\not=0$.
Then \eqref{eq:qreduced} gives
\begin{equation}\label{eq:qreducedproduct}
	1 = \prod_{k=k_0}^m (H_{k}^*)^{e(k)}
= \left\{
\begin{array}{ll}
1 - e(k_0)q^{k_0/N} + O(q^{(k_0+1)/N})
&
\mbox{if $2k_0\not=N$, and},\\
1 - 2e(k_0)q^{k_0/N} + O(q^{(k_0+1)/N})
&
\mbox{if $2k_0=N$.}
\end{array}\right.
\end{equation}
We get $e(k_0)=0$, contradiction.
\end{proof}

\begin{corollary}\label{cor:finiteindex}
Let $S$ be the group of functions
$\prod H_{k}^{e(k)}$ satisfying~\eqref{eq:modulo}.
Then the image of $S$ in $\mathcal{O}(Y^1(N))^*/\QQ^*$
has finite index.
\end{corollary}
\begin{proof}
Proposition~\ref{prop:linearlyindep}
shows that $S$ has rank~$m$.
We have
$$\mathrm{rk}(\mathcal{O}(Y^1(N))^*/\QQ^*)
\leq \#\left(\{\mbox{cusps of $X^1(N)$}\} /
\mathrm{Gal}(\overline{\QQ}/\QQ)\right) - 1.$$
As the right hand side is $m$ by \cite[Definition~1 in \S2]{derickx-vanhoeij},
this proves the result.
\end{proof}
We recover the following 
consequence
of the Manin-Drinfeld
theorem~\cite{manin,drinfeld}, which
states that the cuspidal parts
of modular Jacobians are torsion.
\begin{corollary}
The group $$
\frac{\mathrm{Div}^{0,\mathrm{cusp}}(X^1(N))}{\mathcal{O}(Y^1(N))^*/\QQ^*}$$
of cuspidal divisor classes of $X^1(N)$ is finite.
\end{corollary}
\begin{proof}
% TO KEEP IN COMMENTS:
% Why it also follows from the Manin-Drinfeld theorem:
% The quotient maps injectively to
%    J^1(N)^{\mathrm{cusp}}(\QQ)
% because if an element of Div^{0,cusp} maps
% to the unit element in the Jacobian,
% then it is a principal divisor, and it is also
% cuspidal, so there is some modular unit over $QQbar$
% of which it is a divisor.
% It remains only to show that this modular unit is
% an element of QQbar^* times a modular unit over QQ.
% But that follows from Hilbert 90.
As seen in the proof of Corollary~\ref{cor:finiteindex},
the two groups in the quotient 
both have rank~$m$.
\end{proof}

\subsection{Roots of power series}
\label{sec:gauss}

In Corollary~\ref{cor:finiteindex}, we have shown that every
$f\in\mathcal{O}(Y^1(N))^*$
can be expressed as a product of powers
$c \prod_{k=1}^{m}H_{k}^{e(k)}$
with $e\in \QQ^m$, $c\in\CC^*$ and $m = \lfloor N/2\rfloor$.
The current section is devoted to proving
that the exponent
$e(k)$ is an integer for $k\not=N/2$.
The key idea, taken from Kubert and Lang~\cite{kubert-lang-IV}
is to combine Gauss' lemma for power series with the
fact that $q$-expansions of modular forms have
bounded denominators.

We call a power series $f\in\ZZ[[x]]$ \emph{primitive}
if the ideal generated by its coefficients
is~$(1)$. We then have the following variant of
Gauss' lemma.
\begin{lemma}
\label{lem:gauss}
Let $f, g\in\ZZ[[x]]$ be primitive power series.
Then $fg\in\ZZ[[x]]$ is also primitive.
\end{lemma}
\begin{proof}
Given any prime number~$p$,
take the lowest-order terms of $(f\ \mathrm{mod}\ p)$
and $(g\ \mathrm{mod}\ p)$ (which exist by primitivity). Their product
is a non-zero term of $(fg\ \mathrm{mod}\ p)$, so $p\nmid fg$.
\end{proof}
We say that a Laurent series $f\in\QQ((x))$
has \emph{bounded denominators}
if there is a non-zero $d\in\ZZ$ such that
$df\in\ZZ((x))$.
\begin{corollary}\label{cor:gauss}
Let $f, g\in\QQ[[x]]$ be power series
with
bounded denominators and constant term~$1$.
If $fg$ is in $\ZZ[[x]]$,
then $f, g\in\ZZ[[x]]$.
\end{corollary}
\begin{proof}
Take $a,b\in\ZZ$ such that $af$ and $bg$ are primitive
in $\ZZ[[x]]$.
Then $(ab)(fg)$ is primitive by Lemma~\ref{lem:gauss},
hence $a,b\in\{\pm 1\}$.
\end{proof}
\begin{proposition}[Special case of Lemma 3.1 of
Kubert and Lang~\cite{kubert-lang-IV}]
\label{prop:qexpbounded}
Let $f$ be a modular unit with rational $q$-expansion,
that is, in $\QQ((q^{1/M}))$ for some~$M$.
Then the $q$-expansion has bounded denominators.
\end{proposition}
\begin{proof}
See \cite[Lemma 3.1]{kubert-lang-IV} for the proof,
of which we give a sketch here.
After multiplying
by a suitable power of $\eta^{24}$, the function
becomes a cusp form.
The vector space of cusp forms of given weight
is generated by forms with \emph{integer}
Fourier expansions, hence the result follows.
\end{proof}

For a formal power series $f$ with constant coefficient~$1$
and for $a,b\in\ZZ\setminus\{0\}$,
we define $f^{a/b}$ to be the unique $b$th root of~$f^a$
with constant coefficient~$1$.
For a holomorphic, non-vanishing function $f$ on~$\HH$,
we denote by $f^{a/b}$ any holomorphic $b$th root of~$f^a$.

\begin{proposition}\label{prop:integral}
Let $f$ be a \emph{modular} function of any level
and suppose that we have
$$f = c\prod_{k=1}^m H_{k}^{e(k)}$$
with $e\in\QQ^m$, $c\in\CC^*$ and $m=\lfloor N/2\rfloor$.
Then for all $k$
we have $e(k)\in\ZZ$ if $2k\not=N$ and $e(k)\in\frac{1}{2}\ZZ$
if $2k=N$.
\end{proposition}
\begin{proof}
Taking reduced forms (as defined in the proof of Proposition~\ref{prop:defreducedform})
on both sides,
we get
$$(f^*)^n = \prod_{k=1}^m (H_{k}^*)^{n\cdot e(k)}$$
for some $n$ with $ne\in\ZZ^m$.
The right hand side has integer coefficients,
so
by Proposition~\ref{prop:qexpbounded}
and Corollary~\ref{cor:gauss},
we find that $f^*$ has integer coefficients.

We prove the result by induction on~$k$.
Suppose it is true for all $k<k_0$.
We have
$$f^*\cdot \prod_{k=1}^{k_0-1}(H_{k}^*)^{-e(k)}
= \prod_{k=k_0}^m (H_{k}^*)^{e(k)},$$
and the left hand side has integer coefficients.
By~\eqref{eq:qreducedproduct}, the right hand side
has a coefficient $- e(k_0)$ if $2k_0\not=N$ 
and $-2e(k_0)$ if $2k_0=N$,
hence the result follows.
\end{proof}

\subsection{Using the action}
\label{sec:proof}

Next, we use the action of $\mathrm{SL}_2$.
Recall $m = \lfloor N/2\rfloor$.

\begin{theorem}\label{thm:congruences}
Let $f\in\mathcal{O}(Y^1(N))^*$.
Then $f = c\prod_{k=1}^{m} H_{k}^{e(k)}$,
where $c\in\QQ^*$ and $e\in\ZZ^m$
are uniquely determined by $f$.
Moreover, the vector $e$ satisfies~\eqref{eq:modulo}
that is, it satisfies
\[
\sum_k e(k) \in 12\ZZ\qquad\mbox{and}\qquad
\sum_k k^2 e(k) \in N\gcd(N,2)\ZZ.
\]
\end{theorem}
\begin{proof}
By Corollary~\ref{cor:finiteindex}, we find that $f$ can be written
as $c\prod H_{k}^{e(k)}$ with~$e(k)\in\QQ$.
Here $e$ is uniquely determined
by Proposition~\ref{prop:linearlyindep}.
Moreover, the numbers
$e(1)$, $e(2)$, \ldots, $e(m-1)$ are in~$\ZZ$
by Proposition~\ref{prop:integral}.
Next, we prove~\eqref{eq:modulo},
which also implies $e(m)\in\ZZ$.

Consider the matrix $$M = \left(\begin{array}{cc} 1 & 0\\
1 & 1\end{array}\right)\in\Gamma^1(N).$$
Then 
we have
$f(M\tau) = f(\tau)$, so we inspect
the action of $M$ on the functions~$H_{k}$.
Parts \eqref{item:propertiesofh6} and~\eqref{item:propertiesofh7}
of 
Lemma~\ref{lem:propertiesofh} give
$H_{k}(M\tau) = \exp(2\pi i / 12)^{-1} H_{k}(\tau)$
for this matrix~$M$.
In particular, we get $\sum_k e(k)\in 12\ZZ$.

Next, consider the matrix $$M = \left(\begin{array}{cc} 1 & N\\
0 & 1\end{array}\right)\in\Gamma^1(N).$$
Again we have $f(M\tau) = f(\tau)$, that is,
$f(\tau+N) = f(\tau)$,
which shows that the $q$-expansion of $f$
is in $\CC((q^{1/N}))$.
In the product expansion~\eqref{eq:q},
we consider the leading term
$-i q^{\frac{1}{2}(a_1^2-a_1+\frac{1}{6})}$
(with $a_1=k/N$)
of $H_{k}$.
As the leading term of $f$
is a constant times a power of $q^{1/N}$,
we get
$$\frac{1}{12N^2}\sum_{k=1}^m e(k)(6k^2-6kN+N^2)\in \frac{1}{N}\ZZ.$$

As we already have $\sum e(k)\in 12\ZZ$,
we get
$$\sum_{k=1}^m e(k)(k^2 - kN)\in 2N\ZZ\subset N\ZZ,$$
hence in particular $\sum e(k) k^2\in N\ZZ$.
If $N$ is odd, then this finishes the proof of~\eqref{eq:modulo}.
If $N$ is even, then we get
\begin{align*}
(1-N)\sum_{k=1}^m e(k)k^2 &= \sum_{k=1}^m e(k)(k^2 - k^2 N)\\
&\equiv \sum_{k=1}^m e(k)(k^2-kN)
\equiv 0 \ \mathrm{mod}\ 2N\ZZ,
\end{align*}
and since $N-1$ is coprime to $2N$, this
proves~\eqref{eq:modulo}
and hence $e(m)\in\ZZ$.

It remains to prove $c\in\QQ^*$.
Let $g=f/c$, which is in $\mathcal{O}(Y^1(N))^*$
by Proposition~\ref{prop:hinp}.
Then $c=f/g$ is a constant in $\mathcal{O}(Y^1(N))^*$,
hence is in~$\QQ^*$.
\end{proof}
\begin{proof}[Proof of the main theorems]
Proposition~\ref{prop:express} states exactly
that Theorem~\ref{thm:congruences}
and the rank statement in Proposition~\ref{prop:linearlyindep}
imply
Theorems \ref{thm:mainthm}, \ref{thm:mainthmtwo},
\ref{thm:conjdvh} and \ref{thm:mainthmthree}.
\end{proof}
\begin{remark}
Results similar to Theorem~\ref{thm:congruences},
but assuming integral exponents $e(k)$ and working
with $\Gamma(N)$,
are already known.
These results are insufficient for proving our main
results as they assume that $e(k)$ is integral.

In the special case where $N$ is coprime
to~$6$, they can be used to an alternative
proof of our Theorem~\ref{thm:congruences}
as follows.
If $N$ is odd, then Proposition~\ref{prop:integral}
gives $e\in\ZZ^m$.
For $e\in\ZZ^m$,
Kubert and Lang \cite[Theorem 5.2 and~5.3 on pp.~76--78 in Chapter~3]{kubert-lang-book}
give conditions on $e$ for $f$ to be modular of level~$\Gamma(N)$.
The conditions are complicated,
but if $N$ is coprime to $6$,
then the conditions give exactly \eqref{eq:modulo}, which reproves
Theorem~\ref{thm:congruences} in that case.
\end{remark}

\section{Bonus section}

There are two results that we get almost for free after
all the work that was done towards the main theorem.
We give them here.

\subsection{Ring generators}\label{sec:ring}

In this section,
we give complex analytic functions
that generate the ring
$\mathcal{O}(Y^1(N))$
itself,
instead of its unit group.

\begin{theorem}\label{thm:ringgenerators}
The ring $\mathcal{O}(Y^1(N))$ is generated
as a $\QQ$-algebra by the three functions
$$ b = -t^{3} \frac{H_{2}}{H_{1}},
\quad
c = 
-H_{4} H_{1}^4 H_{2}^{-5},
\quad
d^{-1} = (tH_{1})^{-12},$$
where
$$ t = \frac{H_{1}^2H_{3}}{H_{2}^3}.$$
\end{theorem}
\begin{proof}
By Proposition~\ref{prop:jin},
we have $\mathcal{O}(Y^1(N)) = \QQ[b,c,d^{-1}]$.
We have $b = -p_2$ and $c = p_4/b^5$ by Example~\ref{ex:P}.
Lemma~\ref{lem:express} gives the formulas in terms of
Siegel functions.
\end{proof}

Theorem~\ref{thm:ringgenerators}
is comparable to the main result
of Koo and Yoon~\cite{koo-yoon-ringgenerators}.
Indeed, both give a set of complex analytic
functions that generate the $\QQ$-algebra
$\mathcal{O}(Y^1(N))$,
and through the isomorphism of
Remark~\ref{rem:nonstandard}
also the $\QQ$-algebra
of holomorphic modular functions
on $\Gamma_1(N)\backslash \HH$ with rational
$\QQ$-expansion.
The methods are however completely different.

As for the results themselves, they
are different as well.
First of all, the main result of
\cite{koo-yoon-ringgenerators}
(that is, Theorems 4.5 and~5.2 and Corollary~5.3 of loc.~cit.)
are for
$N=2$, $N=3$ and all~$N$
divisible by $4$, $5$, $6$, $7$ or~$9$,
while our result is for 
all $N\geq 4$.
Second, we give a uniform formula with three generators,
while \cite{koo-yoon-ringgenerators}
has a few different cases, each with
$2$ to~$6$ generators.

\subsection{Expressions in terms of the Jacobi theta function}\label{sec:theta}

In this section, we express the functions $p_k$ in terms of the Jacobi theta function.
This has two applications. First of all, this theta function
can be numerically evaluated efficiently, as in Labrande~\cite{labrande}.
Second, it has a natural generalization to the moduli space of
higher-dimensional abelian varieties
(Riemann theta functions),
potentially opening our results to future higher-dimensional generalisations.

For $c,d\in\RR$, 
the \emph{theta function $\theta[c,d]$ with characteristic} $(c,d)$ is the
function in $z\in\CC$ and $\tau\in\HH$ defined by
\begin{align*}
\theta[c,d](z,\tau)
&= \sum_{n\in\ZZ} \exp\left(\pi i (n+c)^2 \tau + 2\pi i (n+c)(z+d)\right)\\
&= 
                  e^{\pi i c^2 \tau + 2\pi i c(z+d)}\cdot 
          \theta[0,0](z+c\tau+d, \tau).
\end{align*}
We will use a special case, known as the \emph{Jacobi theta function}
$\theta_1 = \theta[\frac{1}{2},\frac{1}{2}] = \theta[-\frac{1}{2},-\frac{1}{2}]$,
that is,
\begin{align*}
\theta_{1}(z,\tau)
&= i \sum_{n\in\ZZ} (-1)^n q^{\frac{1}{2} (n-\frac{1}{2})^2}e^{\pi i(2n-1)z}.
\end{align*}
\newcommand{\T}{T}
\begin{proposition}
Consider the functions $\T_{k}$ given by
\begin{align*}
\T_{k} &= \theta_1\left(\frac{k\tau}{N},\tau\right).
\end{align*}
Then we have for all integers $k$
$$ p_{k} = \left(\frac{\T_1^2\T_3^{\phantom{2}}}{\T_2^3}\right)^{k^2-1}\frac{\T_{k}}{\T_1}
\qquad\mbox{and}\qquad
d =
\left(\frac{\T_1^3\T_3^{\phantom{2}}}{\T_2^3\eta} \right)^{12}.
$$
\end{proposition}
\begin{proof}
Let $\theta_1'(z,\tau) = \frac{d}{dz} \theta_1(z,\tau)$.
Let $\Lambda = 2\omega_1\ZZ+2\omega_3\ZZ$ with $\tau=\omega_3/\omega_1\in\HH$.
Then (6.22) in Theorem 6.5 on page~199 of\cite{markushevich} states
(note that our $q$ is the square of the $q$ in loc.~cit.)
$$\sigma(z,\Lambda) = 2\omega_1 \frac{\theta_1(z/(2\omega_1), \tau)}{\theta_1'(0, \tau)} \exp(\eta_1 z^2 / (2\omega_1)).$$
We choose $\omega_1=\frac{1}{2}$ and $\omega_3=\frac{1}{2}\tau$ to get
$\sigma(z, \tau) = c_1 \exp(c_2z^2) \theta_1(z,\tau)$, where $c_1=\theta_1'(0,\tau)^{-1}$
and $c_2 = \eta_1$ are functions of $\tau$
independent of~$z$.
We apply this to the formulas in Corollary~\ref{cor:sigma} and get
\begin{align*}
\phi_k &:= \frac{\sigma\left(\frac{k\tau}{N},\tau\right)}{\sigma\left(\frac{\tau}{N},\tau\right)^{k^2}}
= c_1^{1-k^2} \frac{T_k}{T_1^{k^2}},\\
u &:= \frac{\phi_3}{\phi_2^3} = c_1 \frac{\T_1^3 T_3^{\phantom{3}}}{\T_2^3},\\
p_k &\phantom{:}= u^{k^2-1}\phi_k = \left(\frac{u}{c_1}\right)^{k^2-1} \frac{\T_{k}}{\T_1^{k^2}}
= \left(\frac{\T_1^2\T_3^{\phantom{2}}}{\T_2^3}\right)^{k^2-1}\frac{\T_{k}}{\T_1},\\
d &\phantom{:}= (2\pi \eta^2u)^{12}.
\end{align*}
This proves the formula for $p_{k}$.
To prove the formula for $d$, it suffices
to prove $2\pi \eta^2 u = \pm u/(c_1\eta)$, or
in other words, $2\pi \eta^3 = \pm \theta_1'(0,\tau)$.
But that is exactly 
the formula for $\theta_1'(0,\tau)$ in the middle of page 210 of
Markushevich~\cite{markushevich}
together with (6.52) on page 211 of loc.~cit.
(In fact, reading further in \cite{markushevich}, we get
that the sign is~$+$,
but we do not need this.)
\end{proof}

\bibliographystyle{plain}
\bibliography{modpol}
\end{document}